\documentclass[11pt, notitlepage, reqno]{amsart}
\usepackage{a4wide}
\usepackage{amsthm,amsmath,amssymb,amscd,mathrsfs,enumerate,amsfonts}

\swapnumbers

\newtheorem{thm}{Theorem}[section]
\newtheorem{prop}[thm]{Proposition}
\newtheorem{lemma}[thm]{Lemma}

\newtheorem*{INTR}{Theorem NZ}

\theoremstyle{definition}
\newtheorem{defn}[thm]{Definition}
\newtheorem{notation}[thm]{Notation}
\newtheorem{remark}[thm]{Remark}
\newtheorem{convention}[thm]{Convention}

\newtheorem*{ack}{Acknowledgment}

\def\F{\mathcal{F}}
\def\en{\mathbb N}
\def\er{\mathbb R}
\def\aa{\mathbb A}

\begin{document}
\author{Martin Rmoutil}
\title{On the nonexistence of a relation between $\sigma$-left-porosity and $\sigma$-right-porosity}
\thanks{The work was supported by the grant GA\v{C}R P201/12/0290.}
\email{mar@rmail.cz}
\address{Charles University, Faculty of Mathematics and Physics, Department of Mathematical\linebreak Analysis, Sokolovsk\'a 83, 186 75 Prague 8, Czech Republic}


\begin{abstract}
Given an arbitrarily weak notion of left-$\langle f \rangle$-porosity and an arbitrarily strong notion of right-$\langle g \rangle$-porosity, we construct an example of closed subset of $\mathbb R$ which is not $\sigma$-left-$\langle f \rangle$-porous and is right-$\langle g \rangle$-porous. We also briefly summarize the relations between three different definitions of porosity controlled by a function; we then observe that our construction gives the example for any combination of these definitions of left-porosity and right-porosity.
\end{abstract}
\maketitle

\section{Introduction}
The topic of this paper originates in the work \cite{3} of R.~J.~Naj\'ares and L.~Zaj\'i\v{c}ek where the following theorem is proved:
\begin{INTR}
There exists a closed set $F\subseteq \er$ which is right-porous and is not $\sigma$-left-porous.
\end{INTR}
We prove the stronger Theorem \ref{hlavniveta} which shows that, as long as we work with a reasonable notion of upper porosity (i.e. porosity defined by $\limsup$ or some equivalent), there is no connection between $\sigma$-left-porosity and right-porosity. The proof is based on the ideas from \cite{3}, but is slightly more technical and contains some new concepts. The main difference lies in our usage of the \emph{multi-expansions} of real numbers (see \ref{defmultiexp}) as opposed to the ordinary decimal expansion used by Naj\'ares and Zaj\'i\v{c}ek. When we have chosen a suitable multi-expansion, however, the rest of the proof follows a scheme identical to that from \cite{3}.

For our method to conveniently go through, we also need to use the right definition of porosity controlled by a function. The one which suits our purpose the best is $[g]$-porosity (see Definition \ref{defhrana}), but that is not a standard notion. For that reason we also give the more common Definitions \ref{defkulata} and \ref{defangle}, together with Proposition \ref{porekviv} which allows us to deduce the validity of our main result in the more standard setting.

Note that porosity controlled by a function can be defined in various ways (see e.g. \cite{5} or \cite{4}) and the notation is not unified. Thus, for example, the notion of $[g]$-porosity we choose to work with is different from that of \cite{5}. However, finding relations between different definitions is usually quite simple and is not our aim in this article.

The proof of Theorem~\ref{hlavniveta} uses the so called Foran Lemma (here Lemma~\ref{ForanLemma}) which is a tool for recognizing non-$\sigma$-porous (in various senses) sets developed by L.~Zaj\'i\v{c}ek from an original idea of J.~Foran. Our proof, however, does not require the most general version, and so we state the definition of Foran system accordingly simplified. The general version for any porosity-like relation in a topologically complete metric space and $G_\delta$-sets (instead of closed sets) can be found in the article \cite{1}.

We shall need Notation~\ref{ZMFci} and Definition~\ref{defhrana} in order to formulate the following definition of Foran System, and the Foran Lemma.

\begin{defn}\label{defForan}
Let $g\in G_3$. We say that $\F\subseteq 2^\er$ is a \emph{Foran system for left-$[g]$-porosity} if the following conditions hold:
\begin{enumerate}[(a)]
\item $\F$ is a nonempty family of nonempty closed sets.
\item For each $F\in\F$ and each open set $G\subseteq\er$ with $F\cap G\neq\emptyset$, there exists a set $F^*\in\F$ such that $F^* \subseteq F\cap G$ and $F$ is left-$[g]$-porous at no point of $F^*$.
\end{enumerate}
\end{defn}

\begin{lemma}\label{ForanLemma}
Let $g\in G_3$ and let let $\F$ be a Foran system for left-$[g]$-porosity. Then no set from $\F$ is $\sigma$-left-$[g]$-porous.
\end{lemma}

\section{Preliminaries to the proof}
This section contains the necessary notation and definitions used in the proof of our main theorem, as well as an important lemma which will allow us to choose a suitable multi-expansion.

\begin{notation}\label{decimalexp}
\mbox{}
\begin{itemize}
\item Let $x\in (0,1)$. As usual, we write $x=0,a_1 a_2 \dots$ if $x=\sum_{i=1}^\infty a_i 10^{-i}$ and $a_i \in \{ 0,1,\dots, 9\}$ for each $i\in \en$. In case $x$ has two expansions, we only consider the one which ends with zeros and then we denote $a_i(x):=a_i$.
\item For an $x\in (0,1)$ we put $l(x):=\sup \{ k\in \en; a_k(x)\neq 0 \}$.
\item We denote by $\aa$ the set of all $x\in (0,1)$ with $l(x)<\infty$.
\end{itemize}
\end{notation}

\begin{lemma}\label{lemmaxn}
Let $f\in G_3$. Then there exists a sequence $\{ x_n \} _{n=1}^\infty \subseteq (0,\infty)\cap \aa$ with the following properties:

\begin{enumerate}[(i)]
\item The sequence $\{ x_n \} _{n=1}^\infty$ is decreasing and $\lim_{n\to \infty}x_n=0$.
\item The sequence of natural numbers $\{ l(x_n) \}_{n=1}^\infty$ is increasing.
\item Let $M\subseteq \er$, $x\in \er$ and let the following proposition be true:
\begin{equation}\label{megavyrok}
(\exists n_0\in {\mathbb N})(\forall n>n_0):\; x+x_n\in M \;\vee\; x+x_{n+1}\in M.
\end{equation}
Then $M$ is not right-$[f]$-porous at $x$.
\end{enumerate}
\end{lemma}

\begin{proof}
Let $f\in G_3$ be defined on the interval $[0,\delta)$. For $x\in [0,\delta)$ define $\alpha(x):=x-f(x)$, and for $x\in \left[ 0, \frac{\delta}{2}\right)$
\begin{equation*}
\beta(x):=\min \alpha \Bigl(\Bigl[x,\frac{\delta}{2} \Bigr] \Bigr).
\end{equation*}
Then $\beta$ is a non-decreasing continuous function with $0<\beta\leq \alpha$ on $\left( 0, \frac{\delta}{2}\right)$. Let us now define the function $g:\left[0,\frac{\delta}{2}\right)\to (0,\infty)$ by the formula
\begin{equation}\label{3lprce2}
g(x):=x-\beta^2(x).
\end{equation}
Obviously $\beta(0)=0$, so we can find a $\delta_1\in\left(0,\frac{\delta}{2}\right)$ such that for each $x\in (0,\delta_1)$ we have $\beta(x)<\frac{1}{2}$ and thus
\begin{equation}\label{3lprce3}
\beta^2(x)<\frac{1}{2}\beta(x) \leq \frac{1}{2}\alpha(x).
\end{equation}
Recall that $g\in G_3$, and hence $g(x)<x$ for all $x\in (0,\delta)$. We obtain:
\begin{gather}
g(g(x))\stackrel{\eqref{3lprce2}}{=} g(x)-\beta^2(g(x))\stackrel{\eqref{3lprce2}}{=} x-\beta^2(x)-\beta^2(g(x))\geq  \notag \\
\geq x-2\beta^2(x) \stackrel{\eqref{3lprce3}}{>} x-\alpha(x) = f(x). \label{odhdktvrposl}
\end{gather}

We shall now construct the sequence $\{ x_n \}_{n=1}^\infty\subseteq(0,\infty)\cap {\mathbb A}$ by induction. First, choose an arbitrary $x_1\in (0,\delta_1)\cap{\mathbb A}$. Now assume we have already chosen the number $x_n$ and take some 
\begin{equation}\label{3lprce4}
y\in \Bigl( g(x_n), g(x_n)+\frac{x_n-g(x_n)}{n+1}\Bigr) \cap {\mathbb A}
\end{equation}
such that $l(y)>l(x_n)$. We set $x_{n+1}:=y$.

We now have a sequence $\{ x_n \}_{n=1}^\infty$ satisfying condition \emph{(ii)} and it remains to be shown it also satisfies \emph{(i)} and \emph{(iii)}.

\paragraph{\bf (i):} It is obvious that the sequence $\{ x_n \}_{n=1}^\infty$ is decreasing and positive; hence, it has a non-negative limit $c$. To obtain a contradiction, assume that $c>0$. Set
\begin{equation*}
m:=\beta^2(c) \qquad\text{and}\qquad M:=\beta^2 \Bigl(\frac{c+\frac{\delta}{2}}{2}\Bigr),
\end{equation*}
and find an $n_0\in \en$ such that
\begin{equation*}
x_{n_0} < \min\Bigl\{c+ m-\frac{M}{n_0+1}\,,\;\frac{c+\frac{\delta}{2}}{2} \Bigr\}.
\end{equation*}
The following estimate gives a contradiction:
\begin{gather*}
x_{n_0+1} \stackrel{\eqref{3lprce4}}{<}g(x_{n_0})+\frac{x_{n_0} - g(x_{n_0})}{n_0+1} \stackrel{\eqref{3lprce2}}{=} x_{n_0}-\beta^2(x_{n_0})+\frac{\beta^2(x_{n_0})}{n_0+1}\leq \\
\leq x_{n_0}-\beta^2(x_{n_0})+\frac{M}{n_0+1} \leq x_{n_0}-m+\frac{M}{n_0+1} < c\,.
\end{gather*}

\paragraph{\bf (iii):} Let us have a set $M\subseteq \er$ and a point $x\in\er$ such that \eqref{megavyrok} is true; let us take some $n_0$ from \eqref{megavyrok}. Then we have $x+x_{n_0+1}\in M$ or $x+x_{n_0+2}\in M$. Now choose an arbitrary $y\in(x,x+x_{n_0+2})$ and find the largest $n_1\in\en$ such that $y\leq x+x_{n_1}$. Now:
\begin{equation*}
f(y-x)\stackrel{\eqref{odhdktvrposl}}{<}g(g(y-x)) \leq g(g(x_{n_1})) \stackrel{\eqref{3lprce4}}{<} x_{n_1+2}<x_{n_1+1} < y-x\,.
\end{equation*}
Setting $I:=(x+f(y-x),y)$ we obtain that $x+x_{n_1+1}\in M\cap I$ or $x+x_{n_1+2}\in M\cap I$ and hence, by the definition, $M$ is not right-$[f]$-porous at $x$.
\end{proof}

\begin{remark}\label{condIV}
It is obvious that the sequence $\{ x_n \}_{n=1}^\infty$ with \emph{(i), (ii), (iii)} from the previous lemma also has the following property.
\emph{
\begin{enumerate}
\item[(iv)] Let $M\subseteq \er$, $x\in \er$ and let the following proposition be true:
\begin{equation}\label{Legavyrok}
(\exists n_0\in {\mathbb N})(\forall n>n_0):\;x-x_n\in M \;\vee\; x-x_{n+1}\in M.
\end{equation}
Then $M$ is not left-$[f]$-porous at $x$.
\end{enumerate}
}
\end{remark}

\begin{defn}\label{defmultiexp}
Let $\{ d_n \}_{n=1}^\infty \subseteq \en$ be a sequence. We denote $D_0:=0$ and then for a number $n\in \en$ we set
\begin{equation}\label{znaceniDn}
D_n:=\sum_{k=1}^n d_k.
\end{equation}
For $x\in (0,1)$ we define the \emph{multi-digit expansion of $x$ with respect to $\{ d_n \}_{n=1}^\infty$} as the sequence $\{ b_n(x) \}_{n=1}^\infty$, where for $n\in \en$,
\begin{equation}\label{defmultirozvojbn}
b_n(x):=\sum_{k=1+D_{n-1}}^{D_n} 10^{D_n-k} a_k(x). 
\end{equation}
The sequence $\{ b_n \}_{n=1}^\infty$ of functions defined on $(0,1)$ by \eqref{defmultirozvojbn} is called the \emph{multi-digit expansion with respect to $\{ d_n \}_{n=1}^\infty$}. 
We will need the symbols
\begin{align*}
C(x,n) & :=\# \left\{ k\in\mathbb N; \, n^2<k\leq(n+1)^2, b_k(x)=10^{d_k}-1\right\},\\
E(x,n) & :=\# \left\{ k\in\mathbb N; \, n^2<k\leq(n+1)^2, b_k(x)\neq 10^{d_k}-1\right\}.
\end{align*}
\end{defn}

\begin{remark}
Let us have a multi-digit expansion $\{ b_n \}_{n=1}^\infty$ with respect to some sequence $\{ d_n \}_{n=1}^\infty\subseteq \en$ and a real number $x\in (0,1)$. Then for each $n\in \en$ we have $b_n(x)\in\left\{0,1,\dots,10^{d_n}-1\right\}$. We can also observe that $x\in(0,1)$ can be written in the form
\begin{equation*}
x=\sum_{n=1}^\infty  b_n(x) \cdot 10^{-D_n}.
\end{equation*}
Note that we still respect the convention from the first point of \ref{decimalexp}; hence $b_n(x)\neq 10^{d_n}-1$ for infinitely many $n\in\en$.
\end{remark}

\begin{convention}
In the following we shall often say briefly \emph{multi-expansion} instead of \emph{multi-digit expansion with respect to $\{ d_n \} _{n=1}^\infty$}, as we will only work with a single sequence $\{d_n\}_{n=1}^\infty$. The functions $b_n$ will also be called \emph{multi-digits}.
\end{convention}

\section{Main result}

\begin{thm}\label{hlavniveta}
Let $f,g\in G_3$. Then there exists a closed set $F\subseteq \er$ which is right-$[g]$-porous and is not $\sigma$-left-$[f]$-porous.
\end{thm}

\begin{proof}
Let the two functions $f, g\in G_3$ be given. First we need to choose the multi-expansion we will work with in the rest of the proof. That is, we need to find a suitable (according to the functions $f$ and $g$) sequence $\{ d_n \}_{n=1}^\infty \subseteq \en$. Using the multi-expansion we will then define a Foran system $\F$ for left-$[f]$-porosity with right-$[g]$-porous elements. According to Lemma \ref{ForanLemma}, any set $F\in\F$ will have the desired properties.

\paragraph{\bf Multi-expansion:} Take the sequence $\{ x_n \}_{n=1}^\infty$ from Lemma \ref{lemmaxn} for $f$; this sequence will remain fixed till the end of the proof. We shall now define the sequence $\{ d_n \}_{n=1}^\infty\subseteq \en$ by induction, respecting the notation from Definition \ref{defmultiexp}. At the same time we also inductively define the multi-expansion $\{ b_n \}_{n=1}^\infty$ by the formula \eqref{defmultirozvojbn}. 

Set $d_1:=\min\{k\in{\mathbb N};\, a_k(x_1)\neq 0\}$. Now suppose we have already chosen the numbers $d_k\in \en$ for all $k\leq n$. Set $k_0:=\max\{k\in{\mathbb N};\, b_n(x_k)\neq 0\}$ and choose a natural number $d_{n+1}$ such that the following conditions hold:

\begin{itemize}
\item $d_{n+1}>d_n$,
\item $g\left(10^{-D_n}\right)>10^{-D_{n+1}}$,
\item $D_{n+1}=\sum_{k=1}^{n+1}d_k > \max\left( \max \left\{ l(x_k);\, b_n(x_k)\neq 0 \right\}, l(x_{k_0+1})\right)$.
\end{itemize}
These conditions can be met, since $x_n\in\aa$ for all $n$ and the sequence tends to zero.

Hence, we have a fixed multi-expansion $\{ b_n \}_{n=1}^\infty$. Observe that each member of the sequence $\{ x_n \}_{n=1}^\infty$ either has only one non-zero multi-digit or it has two consecutive non-zero multi-digits. We will also use the easy fact that, for each $n\in\en$, $\max \{k; b_k(x_{n+1})\neq 0\} \leq \max \{k; b_k(x_n)\neq 0\}+1$.

\paragraph{\bf The sets:} Fix an arbitrary $\alpha\in \left(\frac{1}{2},1 \right)$. For $N\in\en$, $\varepsilon>0$ and multi-digits $B_k\in\{ 0,1,\dots,\linebreak 
10^{d_k}-1\}$, $k\in\{1,\dots,N^2\}$, we define $A(B_1,B_2,\dots,B_{N^2},\varepsilon)$ as the set of all $x\in (0,1)$ which satisfy the following conditions (see Definition~\ref{defmultiexp}):
\begin{enumerate}
\item[\bf(A1)]\qquad $b_1(x)=B_1, b_2(x)=B_2, \dots, b_{N^2}(x)=B_{N^2}$.
\item[\bf(A2)]\qquad $1-\frac{\varepsilon}{n^\alpha}\leq\frac{C(x,n)}{2n+1}<1$ \quad whenever \quad $n\geq N$. 	
\end{enumerate}
We now aim to prove that if
\begin{equation}\label{2rce2}
N>\max\bigl\{ (1+\varepsilon)^{\frac{1}{\alpha}},\varepsilon^{\frac{1}{\alpha-1}} \big\},
\end{equation}
then the set $A(B_1,B_2,\dots,B_{N^2},\varepsilon)$ is nonempty and closed.

\paragraph{\bf The sets are closed:} We will show that each set of the form $A:=A(B_1,\dots,B_{N^2},\varepsilon)$ is closed. To that end, take a convergent sequence $\{ y_m \}_{m=1}^\infty \subseteq A$ with limit $y$. We want to check that $y\in A$.

Condition {\bf(A2)} clearly implies the existence of an $n_0\in\en$ such that, for each $n\geq n_0$, 
\begin{equation}\label{2rce1}
C(x,n)\neq 0 \qquad\text{and}\qquad E(x,n)\neq 0.
\end{equation}
Choose an arbitrary $n\geq n_0$ and find some $m(n)\in\en$ such that
\begin{equation*}
\left| y_{m(n)}-y\right| < 10^{-D}\qquad\text{where}\qquad D:=D_{(n+1)^2+1}.
\end{equation*}
It is now easy to see that \eqref{2rce1} gives $b_k(y_{m(n)}\!)=b_k(y)$ for each $k\leq n^2$. But since the number $n$ was chosen arbitrarily and the corresponding $y_{m(n)}$ is an element of $A$, it follows from the definition of $A$ that $y\in A$.

\paragraph{\bf The sets are nonempty:} Condition {\bf(A2)} from the definition of $A(B_1,\dots,B_{N^2},\varepsilon)$ is equivalent to the statement
\begin{equation*}
C(x,n)\in\Bigl[ \bigl( 1-\frac{\varepsilon}{n^\alpha} 
\bigr)(2n+1),2n+1\Bigr)=:I_n, \quad\text{for each } n\geq N.
\end{equation*}
Assume $n\geq N$; then \eqref{2rce2} gives the following estimates:
\begin{gather}
\left(1-\frac{\varepsilon}{n^\alpha}\right)(2n+1) \stackrel{\eqref{2rce2}}{>} \Bigl(1-\frac{\varepsilon}{1+\varepsilon}\Bigr) \bigl(2(1+\varepsilon)^{\frac{1}{\alpha}}\bigr) \stackrel{(0<\alpha<1)}{>}    2, \label{2rce9}\\
(2n+1)-\left(1-\frac{\varepsilon}{n^\alpha}\right)(2n+1)=\frac{\varepsilon}{n^\alpha}(2n+1)>2\varepsilon n^{1-\alpha} \stackrel{\eqref{2rce2}}{>} 2. \notag
\end{gather}
These two inequalities imply that the interval $I_n$ contains some natural number. Hence, the set $A(B_1,\dots,B_{N^2},\varepsilon)$ is nonemtpy, whenever $N$ fulfills condition \eqref{2rce2}.

\paragraph{\bf Foran system:} Denote by $\F$ the system of all sets of the form $A(B_1,\dots,B_{N^2},\varepsilon)$ with $N$ satisfying \eqref{2rce2}. We already know that $\F$ satisfies the first condition from the definition of Foran system \ref{defForan}. We claim that condition (b) from \ref{defForan} is also true for $\F$.

To prove it, take an arbitrary set $F:=A(B_1,\dots,B_{N^2},\varepsilon)\in \F$ and an open set $G\subseteq \er$ with $F\cap G\neq\emptyset$. Now choose any point $y\in F\cap G$ and a natural number $M$ such that
\begin{gather}\label{2rce4}
M>\max\left\{ N, \bigl( \frac{\varepsilon}{2} \bigr) ^{\frac{1}{\alpha-1}} \right\}\qquad\text{and}\\
\notag
F^*:=A\left(B_1,\dots,B_{N^2},b_{N^2+1}(y),\dots,b_{M^2}(y),\frac{\varepsilon}{2} \right) \subseteq F\cap G.
\end{gather}
It is easy to see that such an $M$, indeed, exists. Hence, we have the set $F^*$ which clearly belongs to $\F$. We now need to show that the set $F$ is left-$[f]$-porous at no point of $F^*$.

Let us fix an arbitrary point $z\in F^*$. As the sequence $\{ x_n \}_{n=1}^\infty$ has the property \emph{(iv)} from Remark \ref{condIV}, it is sufficient to find a $p_0\in\en$ such that, for each $p>p_0$, at least one of the points $z_p^-:=z-x_p$ and $z_{p+1}^-:=z-x_{p+1}$ belongs to $F$. Find a $p_0\in\en$ such that, for each $p>p_0$,
\begin{equation}\label{2rce5}
x_{p}<10^{-D} \qquad\text{where}\qquad D:=D_{(M+1)^2+1}.
\end{equation}
Now choose any $p>p_0$ and $x\in\{ z_p^-, z_{p+1}^- \}$. The estimate \eqref{2rce5} gives that $b_k(x_p)=0$ for all $k\leq (M+1)^2$.	We also know that each $x_n$ has at most two non-zero multi-digits and it follows that, for all $n\in \en$,
\begin{equation}\label{2rce6}
C(x,n)\geq C(z,n)-2.
\end{equation} 
Further, since $z\in F^*$, we have $C(z,M)>0$ (we even know that $C(z,M)>2$) and this yields that, for each $k\leq M^2$, $b_k(x)=b_k(z)$. In particular, $x$ satisfies the condition {\bf(A1)} from the definition of $F=A(B_1,\dots,B_{N^2},\varepsilon)$.

We now turn our attention to condition {\bf(A2)}. The following estimate holds for $n\geq M$:
\begin{equation}\label{2rce8}
\frac{2}{2n+1}=\frac{\varepsilon (\frac{\varepsilon}{2})^{\frac{1-\alpha}{\alpha-1}}}{2n+1} \stackrel{\eqref{2rce4}}{<} \frac{\varepsilon n^{1-\alpha}}{2n+1} < \frac{\varepsilon}{2n^\alpha}
\end{equation}
Thus, for $n\geq M$, we obtain
\begin{equation*}
\frac{C(x,n)}{2n+1} \stackrel{\eqref{2rce6}}{\geq} \frac{C(z,n)-2}{2n+1} \stackrel{(z\in F^*)}{\geq} 1-\frac{\varepsilon}{2n^\alpha}-\frac{2}{2n+1} \stackrel{\eqref{2rce8}}{>} 1-\frac{\varepsilon}{n^\alpha}.
\end{equation*}
It remains to be shown that, for $x=z^-_p$ or $x=z^-_{p+1}$,
\begin{equation*}
E(x,n)\neq 0 \qquad\text{for each }\; n\geq M.
\end{equation*}
To that end, assume that, for some $n\geq M$, $E(z^-_p,n)=0$. Set $k_1:=\max\{m\in\mathbb N;\, b_m(x_p)\neq 0   \}$ and $k_2:=\max\{m\in\mathbb N;\, b_m(x_{p+1})\neq 0   \}$. The way the multi-expansion was constructed implies (as we have observed) that $k_2\in \{ k_1, k_1+1\}$. We also noted that $b_m(x_p)=0$ for all $m\in \en \setminus \{k_1, k_1-1\}$. Taking into account that for each $m\geq M$ we have $C(z,m)>2$, it is easy to see that 
\begin{align}\label{ntyblok}
k_1 =n^2+i \qquad & \text{where }\; i\in\{1,\dots,2n-1\}, \\ \label{ntyblok2}
b_m(z) =10^{d_m}-1 \qquad & \text{for each }\; m\in\{n^2+i+1,\dots,(n+1)^2\}.
\end{align}
We distinguish two cases. In case $k_2=k_1$, we use the property \emph{(ii)} of the sequence $\{ x_n \}_{n=1}^\infty$ which gives that $l(x_{p+1})>l(x_p)$. It follows that $b_{k_1}(x_{p+1})\neq b_{k_1}(x_p)$, and thus $b_{k_1}(z^-_{p+1})\neq b_{k_1}(z^-_p)=10^{d_{k_1}-1}$ which means that $E(z^-_{p+1})\neq 0$. If, on the other hand, $k_2=k_1+1$, then from \eqref{ntyblok} and \eqref{ntyblok2} it follows that $b_{k_1+1}(z^-_{p+1})\neq 10^{d_{k_1+1}-1}$. Again, the conclusion is that $E(z^-_{p+1})\neq 0$. Hence, $\F$ is a Foran system for left-$[f]$-porosity.

\paragraph{\bf Right-$[g]$-porosity:} To conclude the proof of the theorem we shall choose any set $F:=A(B_1,\dots,B_{N^2},\varepsilon) \in \F$ and prove it is right-$[g]$-porous. 

Fix a point $x\in F$. For each $n\in N$, let $m_n$ be the maximum of those $i\in \en$ for which there exist natural numbers $u$ and $v$ such that $v-u=i$,
\begin{equation}\label{2rce14}
n^2\leq u<v \leq(n+1)^2\qquad\text{and}\qquad b_s(x)=10^{d_s}-1 \;\text{ for each } u<s\leq v.
\end{equation}
The following estimate is obvious:
\begin{equation}\label{2rce11}
2n+1-E(x,n)=C(x,n)\leq m	_n(E(x,n)+1)
\end{equation}
Moreover, condition {\bf (A2)} from the definition of $F$ implies that, for each $n\geq N$,
\begin{equation}\label{2rce12}
E(x,n)\leq\frac{\varepsilon(2n+1)}{n^\alpha}.
\end{equation}
Hence, for all $n\geq N$ the following estimate holds:
\begin{gather}
m_n  \stackrel{\eqref{2rce11}}{ \geq}  \frac{2n+1-E(x,n)}{E(x,n)+1} \stackrel{\eqref{2rce12}}{\geq} \frac{(2n+1)(1-\frac{\varepsilon}{n^\alpha})}{\frac{\varepsilon(2n+1)+n^\alpha}{n^\alpha}} \stackrel{\eqref{2rce2}}{\geq} \notag \\
 \stackrel{\eqref{2rce2}}{ \geq}  \frac{n^\alpha(2n+1)\frac{1}{1+\varepsilon}}{\varepsilon(2n+1)+(2n+1)} = \frac{n^\alpha \frac{1}{1+\varepsilon}}{1+\varepsilon}=\frac{n^\alpha}{(1+\varepsilon)^2}=:cn^\alpha. \label{2rce13}
\end{gather}
Now, for each $n$, choose natural numbers $u_n$ and $v_n$ such that
\begin{equation*}
v_n-u_n=m_n \;\text{ and \eqref{2rce14} holds for }\; u=u_n, \; v=v_n.
\end{equation*}
Set $L_n:=D_{v_n-1}$, $K_n:=D_{v_n}$, and define
\begin{equation}\label{2rce17}
y_n:=x+10^{-K_n} \qquad\text{and}\qquad z_n:=x+10^{-L_n}.
\end{equation}
For each $t\in \left( 10^{-K_n}, 10^{-L_n} \right)$, $\min \{k\in{\mathbb N};\, b_k(t)\neq 0\} = v_n$. It is now easy to see that, for each $t\in (y_n, z_n)$ and each $u_n<k \leq v_n-1$, we have $b_k(t)=0$. Hence, for each $t\in (y_n, z_n)$,
\begin{equation}\label{2rce15}
C(t,n)\leq 2n+1-(m_n-1).	
\end{equation}
Since $\alpha>1-\alpha$, we can find an $n_0\in N$ such that, for each natural $n>n_0$,
\begin{equation}\label{2rce16}
2n+2-cn^\alpha<2n+1-2\varepsilon n^{1-\alpha}-\varepsilon n^{-\alpha}=\left(1-\frac{\varepsilon}{n^\alpha}\right)(2n+1).
\end{equation}
Now, for each $n>n_0$, we obtain the estimate
\begin{equation*}	
C(t,n)\stackrel{\eqref{2rce15}}{\leq} 2n+1-(m_n-1) \stackrel{\eqref{2rce13}}{\leq} 2n+2-cn^\alpha \stackrel{\eqref{2rce16}}{<}\left(1-\frac{\varepsilon}{n^\alpha}\right)(2n+1),
\end{equation*}
which implies that $t\notin F=A(B_1,\dots,B_{N^2},\varepsilon,\alpha)$. Thus, for each $n>n_0$, $(y_n,z_n)\cap F=\emptyset$.

It only remains to recall the second condition from the construction of our multi-expansion. It says that
\begin{equation*}
g(z_n-x)=g(10^{-L_n})>10^{-K_n}=y_n-x,
\end{equation*}
and hence we obtain the inclusion $(y_n,z_n)\supseteq (x+g(z_n-x),z_n)$. It follows that the set $F$ is right-$[g]$-porous which concludes the proof.
\end{proof}

\section{Several Definitions of Porosity and Their Relations}
In this last section we briefly discuss the connection between the notion of porosity used in the proof of Theorem~\ref{hlavniveta} and two of the more standard definitions of porosity controlled by a function.

\begin{notation}\label{ZMFci}
The different notions of porosity use the following sets of control functions:
\begin{itemize}
\item $G$ is the set of all increasing continuous functions $f:[0,\delta)\to (0,\infty)$ (where $\delta>0$) with $f(0)=0$.
\item $G_1$ is the set of all $f\in G$ such that $\lim_{x\to 0_+}\frac{f(x)}{x}>0$.
\item $G_2$ is the set of all $f\in G$ such that $f(x)>x$ for each $x\neq 0$ from the domain of $f$.
\item $G_3$ is the set of all $f\in G$ such that $f(x)<x$ for each $x\neq 0$ from the domain of $f$.
\end{itemize}
\end{notation}

\begin{defn}[$(g)$-porosity]\label{defkulata}
Let $M\subseteq \er$ be a set and let $I\subseteq \er$ be an interval. We denote by $\lambda (M,I)$ the length of the largest open subinterval of $I$ which is disjoint from $M$. For $x\in \er$ and $g\in G_1$ we set
\begin{gather*}
p^+_g(M,x):=\limsup_{h\to 0_+}\frac{g(\lambda(M,(x,x+h)))}{h}, \\
p^-_g(M,x):=\limsup_{h\to 0_+}\frac{g(\lambda(M,(x-h,x)))}{h}.
\end{gather*}
We say $M$ is \emph{right-$(g)$-porous at $x$} if $p_g^+(M,x)>0$.
\end{defn}

\begin{defn}[$\langle g \rangle$-porosity]\label{defangle}
Let $M\subseteq \er$ be a set, $r,\delta \in \er$ be such that $0<r<\delta$ and let $g\in G_2$ be defined on $[0,\delta)$. Denote
\begin{equation*}
S^+(g,r,M):=\bigcup \bigl\{ (y-g(\sigma),y); y\in{\mathbb R},\sigma\in (0,r), (y-\sigma,y)\cap M=\emptyset \bigr\}.
\end{equation*}
We say the set $M$ is \emph{right-$\langle g \rangle$-porous at $x$} if 
\begin{equation*}
x\in \bigcap_{0<r<\delta}S^+(g,r,M).
\end{equation*}
\end{defn}

\begin{defn}[${[}g{]}$-porosity]\label{defhrana}
Let $M\subseteq \er$, $x\in\er$ and $g\in G_3$. We say the set $M$ is \emph{right-$[g]$-porous at $x$} if there exists a sequence $\{ \alpha_n \} _{n=1}^\infty \subseteq\er$ such that the following conditions are satisfied:
\begin{enumerate}[(i)]
\item The sequence $\{ \alpha_n \} _{n=1}^\infty$ is decreasing and it tends to $0$.
\item For each $n\in \en$ we have $(x+g(\alpha_n),x+\alpha_n)\cap M=\emptyset$.
\end{enumerate}
\end{defn}

Of course, definitions \ref{defkulata}, \ref{defangle}, and \ref{defhrana} can be stated ``for the left side'' in the obvious symmetrical way.

\begin{defn}
Let $M\subseteq \er$, let $x\in \er$, and assume $V$ is one of the symbols $(g)$, $\langle g \rangle$, $[g]$. We say the set $M$ is
\begin{itemize}
\item \emph{$V$-porous at $x$} if it is left-$V$-porous at $x$ or right-$V$-porous at $x$,
\item \emph{$V$-porous} if it is $V$-porous at each of its points,
\item \emph{right-$V$-porous} if it is right-$V$-porous at each of its points,
\item \emph{$\sigma$-$V$-porous} if it is a countable union of $V$-porous sets,
\item \emph{$\sigma$-right-$V$-porous} if it is a countable union of right-$V$-porous sets.
\end{itemize}
\end{defn}

The following proposition deals (to some extent) with the relation of the previously defined notions of porosity. In order to formulate it in a briefer form, we need a special notation:

\begin{notation}
\mbox{}
\begin{itemize}
\item If $g\in G_1$, we write $^1\! g$ instead of $(g)$.
\item If $g\in G_2$, we write $^2\! g$ instead of $\langle g \rangle$.
\item If $g\in G_3$, we write $^3\! g$ instead of $[g]$.
\end{itemize}
\end{notation}

\begin{prop}\label{porekviv}
Let $i,j\in \{ 1,2,3 \}$ and let $f\in G_i$. Then:
\begin{enumerate}[(i)]
\item If $j\neq 1$, then there is a function $g\in G_j$ such that whenever $M\subseteq \er$ is right-$\,^j\! g$-porous (resp. left-$\,^j\! g$-porous), then $M$ is right-$\,^i \! f$-porous (resp. left-$\,^i \! f$-porous).
\item There exists a function $h\in G_j$ such that whenever $M\subseteq \er$ is right-$\, ^i\! f$-porous (resp. left-$\,^i \! f$-porous), then $M$ is right-$\, ^j \! h$-porous (resp. left-$\,^j\! h	$-porous).
\end{enumerate}
\end{prop}

We omit the proof as it is completely straightforward but quite long. The sole purpose of this proposition in the article is to show that Theorem \ref{hlavniveta}, which is formulated with the non-standard $[f]$-porosity, is also true with other kinds of porosity controlled by a function. (In fact, to see this, we only need to consider the case $j=3$.)

\begin{remark}
The part \emph{(i)} of Proposition \ref{porekviv} says that, for any given type of porosity controlled by a given function, a function $g\in G_j$ can be found such that the notion of right-$\, ^j \! g$-porosity ``is stronger''. We can not include the case $j=1$, because the $^1 \! g$-porosity is the strongest possible for $g(x)=x$, and that is the ordinary upper porosity. For the other two cases, however, we can obtain (in some sense) arbitrarily strong notions of porosity.
\end{remark}

\begin{ack}
I would like to thank Prof. Lud\v{e}k Zaj\'i\v{c}ek for suggesting the topic of this article and several useful remarks.
\end{ack}


\begin{thebibliography}{99}

\bibitem{3}R.J. Naj\'ares, Zaj\'i\v{c}ek, \emph{A $\sigma$-porous set need not be $\sigma$-bilaterally porous}, Comment. Math. Univ. Carolin. {\bf 35} (1994), 697-703.
\bibitem{4} L. Zaj\'i\v{c}ek, \emph{Sets of $\sigma$-porosity and sets of $\sigma$-porosity $(q)$}, \v{C}asopis P\v{e}st. Mat. {\bf 101} (1976), 350-359.
\bibitem{5} \label{7} L. Zaj\'i\v{c}ek, \emph{Porosity and $\sigma$-porosity}, Real Anal. Exchange {\bf 13} (1987/88), 314-350.
\bibitem{1} \label{1} L. Zaj\'i\v{c}ek, \emph{Products of non-$\sigma$-porous sets and Foran systems}, Atti Semin. Mat. Fis. Univ. Modena, {\bf 44} (1996), 497-505.
\bibitem{2} \label{2} L. Zaj\'i\v{c}ek, \emph{On $\sigma$-porous sets in abstract spaces}, Abstr. Appl. Anal. {\bf 5} (2005), 509-534.
\end{thebibliography}
\end{document}